\documentclass[11pt]{article}
\usepackage[latin1]{inputenc}
\usepackage[margin=1in]{geometry}
\usepackage{amsmath,amssymb,amstext}
\usepackage{amsthm}
\usepackage{cleveref}
\usepackage{authblk}
\usepackage{enumerate}
\usepackage{dirtytalk}
\usepackage{tikz}
\usepackage{thmtools}
\usepackage{thm-restate}
\usetikzlibrary{shapes}

\usepackage{tikz}
\usepackage{thmtools}
\usepackage{thm-restate}
\usetikzlibrary{shapes}

\newtheorem{thm}{Theorem}[section]

\newtheorem{claim}{Claim}
\newtheorem{question}[equation]{Question}

\newtheorem*{ack}{Acknowledgements}

\numberwithin{equation}{section}
\theoremstyle{definition} 
\newtheorem{definition}[thm]{Definition}

\definecolor{asparagus}{rgb}{0.53, 0.66, 0.42}
\definecolor{cerulean}{rgb}{0.0, 0.48, 0.65}
\definecolor{cornellred}{rgb}{0.7, 0.11, 0.11}
\definecolor{darklavender}{rgb}{0.45, 0.31, 0.59}
\definecolor{darkslateblue}{rgb}{0.28, 0.24, 0.55}
\definecolor{burntorange}{rgb}{0.8, 0.33, 0.0}
\date{}

\begin{document}


\author{Evelyne Smith-Roberge\thanks{We acknowledge the support of the Natural Sciences and Engineering Research Council of Canada (NSERC) [CGSD3 Grant No. 2020-547516]. \\
\hphantom{|} Cette recherche a \'{e}t\'{e} financ\'{e}e par le Conseil de recherches en sciences naturelles et en g\'{e}nie du Canada (CRSNG) [CGSD3 Grant No. 2020-547516].} }
\affil{Dept.~of Combinatorics and Optimization, University of Waterloo, 
Canada. \\ \texttt{e2smithr@uwaterloo.ca}}

\title{On the choosability with separation of planar graphs and its correspondence colouring analogue}
\date{\today}
\maketitle
\date{}

\begin{abstract}
A list assignment $L$ for a graph $G$ is an $(\ell,k)$-list assignment if $|L(v)|\geq \ell$ for each $v \in V(G)$ and $|L(u) \cap L(v)| \leq k$ for each $uv \in E(G)$. We say $G$ is $(\ell,k)$-choosable if it admits an $L$-colouring for every $(\ell, k)$-list assignment $L$. We prove that if $G$ is a planar graph with $(4,2)$-list assignment $L$ and for every triangle $T \subseteq G$ we have that $|\bigcap_{v \in V(T)} L(v)| \neq 2$, then $G$ is $L$-colourable. In fact, we prove a slightly stronger result: if $G$ contains a clique $H$  such that $V(H) \cap V(T) \neq \emptyset$ for every triangle $T \subseteq G$ with $|\bigcap_{v \in V(T)} L(v)| = 2$, then $G$ is $L$-colourable. Additionally, we give a counterexample to the correspondence colouring analogue of $(4,2)$-choosability for planar graphs. \end{abstract}
\maketitle

\section{Introduction}
List colouring is a generalization of vertex colouring introduced in the late '70s by Erd\H{o}s, Rubin, and Taylor \cite{erdos1979choosability} and, independently, by Vizing \cite{vizing}. A \emph{$k$-list assignment} $L$ for a graph $G$ is a function that assigns to each vertex $v \in V(G)$ a list $L(v)$ of at least $k$ colours. We say $G$ is $L$-colourable if $G$ admits a proper colouring $\phi$ with $\phi(v) \in L(v)$ for all $v \in V(G)$, and that $G$ is \emph{$k$-choosable} if it admits an $L$-colouring no matter the $k$-list assignment $L$.

 In '98, Kratochv{\'\i}l, Tuza, and Voigt \cite{kratochvil1998brooks} introduced a variation of list colouring defined below.
\begin{definition}
Let $\ell$ and $k$ be positive integers. An $(\ell, k)$-list assignment for a graph $G$ is an $\ell$-list assignment $L$ such that for each $uv \in E(G)$, we have $|L(u) \cap L(v)| \leq k$.
\end{definition}

When $\ell = k$, we recover ordinary choosability; when $\ell > k$, the concept is referred to in the literature as \emph{choosability with separation}. In \cite{kratochvil1998brooks}, Kratochv{\'\i}l, Tuza,  and Voigt showed that every planar graph is $(4,1)$-choosable. This was later strengthened by Kierstead and Lidick{\`y} \cite{kierstead2015choosability}, who showed that the same result holds if we allow an independent set of vertices to have lists of size three. There exist planar graphs that are not $(4,4)$-choosable: one such graph was given by Voigt in \cite{voigt1993list}. Moreover, there exist planar graphs that are not $(4,3)$-choosable: an example was given by Mirzakhani in \cite{mirzakhani1996small}, and another example is given in Section 3. The following question, however, remains open.

\begin{question}\label{42conj}
Is every planar graph $(4,2)$-choosable?
\end{question}

If we further restrict the structure of the underlying graph, more results in this area are known. For instance, Zhanar et al. showed in \cite{berikkyzy20174} that if $t \in \{5,6,7\}$ and $G$ is a graph that does not contain a chorded $t$-cycle\footnote{A \emph{chorded $t$-cycle} is a cycle of length $t$ with an additional edge between its vertices.}, then $G$ is $(4,2)$-choosable. 

In a similar vein as Question \ref{42conj}, \v{S}krekovski asked the following \cite{skrekovski2001note}.

\begin{question}\label{31conj}
Is every planar graph $(3,1)$-choosable?
\end{question}

Moreover, \v{S}krekovski showed that there exist planar graphs that are not $(3,2)$-choosable. Though Question \ref{31conj} remains unsolved, there has been a great deal of partial progress. As in the case of $(4,2)$-choosability, most results are obtained by further limiting the structure of the underlying planar graph.  For instance, Kratochv{\'\i}l, Tuza, and Voigt \cite{kratochvil1998brooks} showed that triangle-free planar graphs are (3,1)-choosable.  Choi, Lidick{\`y}, and Stolee show in \cite{choi2016choosability} that planar graphs with either no 4-cycles or no cycles of length 5 and 6 are (3,1)-choosable. Chen, Fan, Wang, and Wang show in \cite{chen2017sufficient}  that planar graphs with neither 6-cycles nor adjacent 4- and 5-cycles are (3, 1)-choosable.
Moreover, Chen, Ko-Wei, and Wang show in \cite{chen2018choosability} that planar graphs without 4-cycles adjacent to cycles with lengths in $\{3,4\}$ are (3, 1)-choosable.

Our first result concerns the $(4,2)$-choosability of planar graphs. Instead of further restricting the class of graphs, we further restrict the list assignment. To that end, we give the following convenient definition.

\begin{definition}
Let $G$ be a planar graph with $(4,2)$-list assignment $L$. An \emph{offensive triangle} $T$ is a triangle $T \subseteq G$ with $|\bigcap_{v \in V(T)} L(v)| = 2$.\end{definition}

Our first result is the following.

\begin{restatable}{thm}{main}\label{thm-main}
Let $G$ be a planar graph with $(4,2)$-list assignment $L$. If $G$ contains a clique that intersects every offensive triangle, then $G$ is $L$-colourable.
\end{restatable}

It is our hope that, should the answer to Question \ref{42conj} be \say{no}, Theorem \ref{thm-main} and similar results will help guide the construction of a witness. As an immediate corollary, we obtain the following theorem, which is the main result of this paper.
\begin{thm}\label{notriangles}
If $G$ is a planar graph with $(4,2)$-list assignment $L$ and $G$ contains no offensive triangles, then $G$ is $L$-colourable.
\end{thm}

We prove Theorem \ref{thm-main} via a stronger theorem that further restricts the lists of vertices in the plane graph \textemdash and in particular, the lists of the vertices in the outer face boundary. The proof is inspired by Thomassen's famous and elegant proof of the 5-choosability of planar graphs \cite{thomassen5LC}.

\begin{restatable}{thm}{inductive}\label{thm-inductive}
Let $G$ be a plane graph, and let $C$ be the subgraph of $G$ whose edge-set and vertex-set are precisely those of the boundary walk of the outer face of $G$. Let $P\subseteq C$ be a path of length at most one. Let $L$ be a list assignment for $G$ where:
\begin{itemize}
    \item $|L(v)| \geq 4$ for all $v \in V(G) \setminus V(C)$,
    \item $|L(v)| \geq 1$ for all $v \in V(P)$,
    \item $|L(v)| \geq 3$ for all $v \in V(C) \setminus V(P)$, and
    \item $|L(v) \cap L(u)| \leq 2$ for each $uv \in E(G)$.  
\end{itemize}
 If $G$ contains no offensive triangles, then every $L$-colouring of $P$ extends to an $L$-colouring of $G$.
\end{restatable}

We note that Theorem \ref{notriangles} is also an immediate corollary of Theorem \ref{thm-inductive}.

Our next main results concern the \emph{correspondence colouring} (also known as \emph{DP-colouring}) versions of Questions \ref{42conj} and \ref{31conj}.  Correspondence colouring is a generalization of list colouring first introduced by Dvo{\v{r}}{\'a}k and Postle in \cite{dvovrak2018correspondence}. A definition is given below.

\begin{definition}
 Let $G$ be a graph. A \emph{correspondence assignment $(L,M)$ for $G$} is a list assignment $L$ together with a function $M$ that assigns to every edge $e = uv \in E(G)$ a partial matching $M_{e}$ between $\{u\}\times L(u)$ and $\{v\}\times L(v)$.
An $(L,M)$-colouring of $G$ is a function $\varphi$ that assigns to each vertex $v \in V(G)$ a colour $\varphi(v) \in L(v)$ such that for every $e = uv \in E(G)$, we have that $(u, \varphi(u))(v, \varphi(v)) \not \in M_e$. We say that $G$
is $(L,M)$-colourable if such a colouring exists.
\end{definition}

Informally, we think of correspondence colouring as a version of list colouring where the vertices may disagree on the meanings of colours: in list colouring, if a vertex is coloured blue, then its neighbours must not be coloured blue. In the correspondence colouring framework, if a vertex is coloured blue, that might enforce that one of its neighbours not be coloured red; that another not be coloured green; and so on.

We define the correspondence analogue of $(\ell,k)$-choosability as follows.

\begin{definition}
Let $G$ be a graph and $(L,M)$, a correspondence assignment for $G$. We say $(L,M)$ is an \emph{$(\ell, k)$-correspondence assignment} if $L$ is an $\ell$-list assignment and for every edge $e \in E(G)$, we have $|M_e|\leq k$. We say $G$ is \emph{$(\ell, k)$-correspondence choosable} if $G$ is $(L,M)$-colourable for every $(\ell,k)$-correspondence assignment $(L,M)$.
\end{definition}

 Given an $(\ell,k)$-correspondence assignment, we have (as in the list colouring version) that for each edge $uv \in E(G)$ there are at most $k$ colours in $L(u)$ that rule out a colour option for $v$.  In \cite{dvovrak2021single}, Dvo{\v{r}}{\'a}k,  Esperet,  Kang, and Ozeki further investigate the notion of $(k,1)$-correspondence choosability of embedded graphs, and generalize the study to multigraphs.

It is natural to ask whether theorems concerning list colouring with separation carry over to the realm of correspondence colouring.  The proof given in \cite{kratochvil1998brooks} that every planar graph is $(4,1)$-choosable also holds for correspondence colouring, and any counterexample for planar $(4,3)$-choosability is of course a counterexample for planar $(4,3)$-correspondence choosability.  Moreover, in \cite{hell2008adaptable}, Zhu and Hell give a construction of a planar graph that has adaptable chromatic number 4, and as such is not $(3,1)$-correspondence choosable. Thus for planar graphs the only remaining case for correspondence choosability with separation is the question of $(4,2)$-correspondence choosability.

In this paper, we show that the answer to the correspondence colouring analogue of Question \ref{42conj} is \say{no}.
 
\begin{restatable}{thm}{x42}\label{ctx42}
There exist planar graphs that are not $(4,2)$-correspondence choosable.
\end{restatable}


In addition to being interesting in its own right, Theorem \ref{ctx42} helps inform the structure of potential proofs of the $(4,2)$-choosability of planar graphs. In particular, any such proof will have to take into consideration more than merely the number of colours available to each vertex, as otherwise it would also hold for correspondence colouring.  

Section 2 will be devoted to proofs of Theorems \ref{thm-inductive} and \ref{thm-main}. In Section 3, we will give an explicit construction of a graph and correspondence assignment proving Theorem \ref{ctx42}.

\section{Proving Theorems \ref{thm-inductive} and \ref{thm-main}}\label{minctx}
In this section, we prove Theorems \ref{thm-inductive} and \ref{thm-main}. In what follows, we denote by $N_G(v)$ the set of vertices in a graph $G$ that are adjacent to a vertex $v \in V(G)$.

We begin with Theorem \ref{thm-inductive}, restated for convenience below. 

\inductive*

\begin{proof}

Suppose not. Let $G$ be a counterexample to Theorem \ref{thm-inductive} that is minimum with respect to $|V(G)|+|E(G)|$; and, subject to that, with respect to $\sum_{v \in V(G)} |L(v)|$. It is easy to verify that the theorem holds if $|V(G)| \leq 2$, so we may assume that $|V(G)| \geq 3$. By our choice of $G$, we may assume further that every vertex $v \in V(G) \setminus V(C)$ has $|L(v)| =4$; that every vertex $v \in  V(C)$ has $|L(v)| \leq 3$; that $|V(P)| = 2$; and that $|L(v)| = 1$ for each $v \in V(P)$. Finally, we may assume that $P$ has an $L$-colouring as otherwise there is nothing to prove.

We proceed via a series of claims establishing the structure of $G$ and the lists of a subset of the vertices in $V(C)$.

\begin{claim}\label{2-conn}
$G$ is 2-connected.
\end{claim}
\begin{proof}
Suppose not. If $G$ is disconnected, then by the minimality of $G$ each component of $G$ admits an $L$-colouring, a contradiction. Thus we may assume that $G$ is connected and contains a cut-vertex $u$. Let $G_1$ and $G_2$ be connected subgraphs of $G$ with $G_1 \cap G_2 = u$; with $V(G) \setminus V(G_i) \neq \emptyset$ for each $i \in \{1,2\}$; and with $G_1 \cup G_2 = G$. Without loss of generality, we may assume that $P \subseteq G_1$. Let $\phi$ be an $L$-colouring of $P$. By the minimality of $G$, we have that $\phi$ extends to an $L$-colouring of $G_1$. Let $L'$ be a list assignment for $G_2$ with $L'(u) = \{\phi(u)\}$ and $L'(v) = L(v)$ for all $v \in V(G_2) \setminus \{u\}$. By the minimality of $G$, we have that $G_2$ admits an $L'$-colouring $\phi'$. But then $\phi \cup \phi'$ is an $L$-colouring of $G$, a contradiction.  
\end{proof}

It follows from Claim \ref{2-conn} that $C$ is a cycle. Let $C = v_0v_1\cdots v_{k-1}v_0$, where $P=v_0v_1$. In what follows, all indices are understood modulo $k$. 

\begin{claim}\label{chordless}
$C$ is chordless.
\end{claim}
\begin{proof}
Suppose not, and let $v_iv_j$ be a chord of $C$. Let $G_1$ and $G_2$ be connected subgraphs of $G$ with $G_1 \cap G_2 = v_iv_j$; with $V(G) \setminus V(G_i) \neq \emptyset$ for each $i \in \{1,2\}$; and with $G_1 \cup G_2 = G$. Without loss of generality, we may assume that $P \subseteq G_1$. By the minimality of $G$, we have that $G_1$ admits an $L$-colouring $\phi$. Let $L'(v) = \{\phi(v)\}$ for each $v \in \{v_i, v_j\}$, and let $L'(v) = L(v)$ for each $v \in V(G_2) \setminus \{v_i, v_j\}$. By the minimality of $G$, we have that $G_2$ admits an $L$-colouring $\phi'$. But $\phi \cup \phi'$ is an $L$-colouring of $G$, a contradiction.
\end{proof}

Without loss of generality, we may assume that $L(v_0) = \{1\}$ and $L(v_1) = \{2\}$.
\begin{claim}
$2 \in L(v_2)$.
\end{claim}
\begin{proof}
If not, then $G-v_1v_2$ contradicts our choice of $G$.
\end{proof}

Let $L(v_2) = \{2, a, b\}$, where possibly $1 \in \{a,b\}$. 

\begin{claim}\label{abinLv4}
 $\{a,b\} \subseteq L(v_3)$.
\end{claim}
\begin{proof}
Suppose not. Let $c$ be a colour in $\{a,b\} \setminus L(v_3)$ (where recall the index is taken modulo $k$). Let $G'$ be the graph obtained from $G$ by deleting $v_2$, and let $L'$ be the list assignment for $G'$ with $L'(v) = L(v)$ for all $v \in V(G') \setminus N_G(v_2)$, and $L'(v) = L(v)\setminus \{c\}$ for all $v \in N_G(v_2)$. Let $C'$ be the subgraph of $G'$ whose edge-set and vertex-set are precisely those of the boundary walk of the outer face of $G'$. Note that by Claim \ref{chordless} and our choice of $c$, we have that $|L(v_3)| = |L'(v_3)|$ and $|L(v_2)| = |L'(v_2)|$. Moreover, for every vertex $v \in N_G(v_2) \setminus \{v_3, v_1\}$, we have that $|L'(v)| \geq 3$. Thus $G'$ and $L'$ satisfy the hypotheses of Theorem \ref{thm-inductive}, and so by the minimality of $G$, we have that $G'$ admits an $L'$-colouring $\phi$. But $\phi$ extends to an $L$-colouring of $G$ by setting $\phi(v_2) = c$, a contradiction.
\end{proof}
Since $|L(v_0)| = 1$, Claim \ref{abinLv4} implies the following.
\begin{claim}\label{kgeq4}
$k \geq 3$.
\end{claim}

By Claim \ref{abinLv4} we may assume that $L(v_3) = \{a,b,d\}$ for some colour $d$. Note that since $L$ is a restriction of a $(4,2)$-list assignment, $d \not \in L(v_2)$.
\begin{claim}\label{dinv4}
$d \in L(v_4)$.
\end{claim}
\begin{proof}
Suppose not. Let $G' = G-v_3$, and let $L'$ be the list assignment for $G'$ with $L'(v) = L(v)$ for all $v \in V(G') \setminus N_G(v_3)$, and $L'(v) = L(v) \setminus \{d\}$ for all $v \in N_G(v_3)$. Let $C'$ be the subgraph of $G'$ whose edge-set and vertex-set are precisely those of the boundary walk of the outer face of $G'$. By Claim \ref{chordless} and the fact that $d \not \in L(v_2) \cup L(v_4)$, we have that $|L'(v)| \geq 3$ for all $v \in V(C') \setminus V(P)$. By the minimality of $G$, it follows that $G'$ admits an $L'$-colouring $\phi$. But $\phi$ extends to an $L$-colouring of $G$ by setting $\phi(v_3) = d$, a contradiction.
\end{proof}

Since $L$ is a restriction of a $(4,2)$-list assignment and $d \in L(v_4)$ by Claim \ref{dinv4}, it follows that there exists a colour in $\{a,b\}$ that is not in $L(v_4)$. Without loss of generality, we may assume $a \not \in L(v_4)$. Let $\phi$ be a colouring of $v_2, v_3$ with $\phi(v_3) = a$ and $\phi(v_2) = b$.

Let $G' = G-\{v_2, v_3\}$, and let $L'$ be the list assignment for $G'$ with $L'(v) = L(v)$ for all $v \in V(G') \setminus (N_G(v_2) \cup N_G(v_3))$, and $L'(v) = L(v) \setminus \{\phi(u): u\in \{v_2,v_3\} \cap N_G(v)\}$ for all $v \in N_G(v_2) \cup N_G(v_3)$. Let $C'$ be the subgraph of $G$ whose edge-set and vertex-set are precisely those of the boundary walk of the outer face of $G'$. Note that since $G$ contains no offensive triangles and $\{a,b\} \subseteq L(v_2) \cap L(v_3)$, it follows that no vertex $v \in V(C') \setminus V(C)$ has $|L'(v)|\leq 2$. Moreover, since $a \not \in L(v_4)$, every vertex $v \in V(C') \cap V(C)$ satisfies $|L'(v)| = |L(v)|$. Thus $G'$ and $L'$ satisfy the hypotheses of Theorem \ref{thm-inductive}, and so by the minimality of $G$ we have that $G'$ admits an $L'$-colouring $\phi'$. But $\phi' \cup \phi$ is an $L$-colouring of $G$, a contradiction.
\end{proof}

Note that since the list assignment described in the statement of Theorem \ref{thm-inductive} is a restriction of the list assignment described in Theorem \ref{notriangles}, this immediately implies Theorem \ref{notriangles}.

Moreover, we note that a nearly identical proof can be used to show the following.

\begin{thm}
Let $G$ be a planar graph and $(L,M)$, a $(4,2)$-correspondence assignment for $G$. If for every triangle $xyzx$ in $G$ we have that $|M_{xy}| \cup |M_{yz}| \cup |M_{zx}| <4$, then $G$ is $(L,M)$-colourable.
\end{thm}

We now prove Theorem \ref{thm-main}, restated below for convenience.

\main*

\begin{proof}
Let $G$ be a plane graph with $(4,2)$-list assignment $L$, and suppose $G$ contains a clique $H$ that intersects every offensive triangle in $G$. Note that since $G$ is planar, $|V(H)| \leq 4$. We break into cases depending on the number of vertices in $V(H)$.
\vskip 2mm
\noindent
\textbf{Case 1: $|V(H)| \leq 3$.}  Let $v_1 \in V(H)$, and let $\phi$ be an $L$-colouring of $H$.  Let $G' = G-v_1$, and let $L'$ be the list assignment obtained from $L$ by setting $L'(v) = L(v) $ for all $v \in V(G') \setminus N_G(v_1)$, and $L'(v) = L(v) \setminus \{\phi(v_1)\}$ for all $v \in N_G(v_1) \setminus V(H)$, and $L'(v) = \{\phi(v)\}$ for all $v \in V(H) \setminus \{v_1\}$. Note that there exists an embedding of $G'$ where every vertex in $N_G(v_1)$ is in the boundary walk of the outer face of $G'$. In what follows, we fix such an embedding. Note further that since every vertex in $V(H) \cap V(G')$ has a list of size one, it follows that $G'$ contains no offensive triangles. 

First suppose that $H-v_1$ is entirely contained in the boundary walk of the outer face of $G'$. Then $G'$ satisfies the hypotheses of Theorem \ref{thm-inductive}, and so admits an $L'$-colouring $\phi'$. Then $\phi' \cup \phi$ is an $L$-colouring of $G$, as desired.

Thus we may assume that $H-v_1$ is not entirely contained in the boundary walk of the outer face of $G'$. Since $V(H)-\{v_1\}$ is contained in the boundary walk of the outer face of $G'$, it follows that $H-v_1$ contains two vertices \textemdash say $v_2$ and $v_3$ \textemdash and that the edge $v_2v_3$ is not contained in the boundary walk of the outer face of $G'$. Thus $v_2v_3$ is a chord in a cycle contained in the boundary walk of the outer face of $G'$. Note the chord $v_2v_3$ separates $G'$ into two graphs $G_1$ and $G_2$ with $G_1 \cup G_2 = G$, and $G_1 \cap G_2 = v_2v_3$. Moreover, $v_2v_3$ is contained in the outer face boundary walk of each of $G_1$ and $G_2$. Thus by Theorem \ref{thm-inductive}, $G_1$ and $G_2$ each admit an $L'$-colouring $\phi_1$ and $\phi_2$, respectively; and so $\phi_1\cup \phi_2 \cup \phi$ is an $L$-colouring of $G$, as desired.
\vskip 2mm
\noindent
\textbf{Case 2: $|V(H)| = 4$.} Let $v_1, v_2, v_3$ and $v_4$ be the vertices of $H$, and let $f_1, f_2, f_3, f_4$ be the faces of $H$ where $f_4$ is the infinite face and is incident to $v_1, v_2$, and $v_3$. Let $\phi$ be an $L$-colouring of $G$, and for $i \in \{1,2,3,4\}$ let $G_i$ be the subgraph of $G$ induced by the vertices in $f_i$ and its boundary. For $i = \{1,2,3\}$, let $G_i' = G_i - v_4$, and let $L'$ be a list assignment obtained from $L$ by setting $L'(v) = L(v) $ for all $v \in V(G_i) \setminus N_G(v_4)$, and $L'(v) = L(v) \setminus \{\phi(v_4)\}$ for all $v \in N_G(v_4) \setminus V(H)$, and $L'(v) = \{\phi(v)\}$ for all $v \in V(H) \setminus \{v_4\}$. Note that for each $i \in \{1,2,3\}$, we have that $G_i'$ contains no offensive triangles. Moreover, $G_i'$ satisfies the hypotheses of Theorem \ref{thm-inductive} and thus admits an $L'$-colouring $\phi_i$.  
By an identical argument to that given in Case 1, $G_4$ admits an $L'$-colouring $\phi_4$. By construction $\phi \cup \phi_1 \cup \phi_2 \cup \phi_3 \cup \phi_4$ is an $L$-colouring of $G$, as desired. 
\end{proof}

\section{Counterexamples for Correspondence Colouring}\label{ctx}
In this section, we prove Theorem \ref{ctx42} by constructing a graph $G$ that is not $(4,2)$-correspondence colourable for a given $(4,2)$ correspondence assignment $(L,M)$. Note that $G$ and $L$ together are also an example of a graph that is not $(4,3)$-choosable. For each edge $uv \in E(G)$ the edges in the matchings $M_{uv}$ will be of the form $M_{uv} = \{(u,c_1)(v,c_1), \dots, (u,c_k)(v,c_k)\}$. For that reason, we will write $M_{uv} = \{c_1, \dots, c_k\}$ instead of $M_{uv} = \{(u,c_1)(v,c_1), \dots, (u,c_k)(v,c_k)\}$. Informally: if a colour in $L(v)$ is matched to a colour in $L(u)$,  both vertices will have the same name for this colour (though the converse crucially does not hold). 







We note the gadget used below (as well as its list assignment) was also used by Voigt and Wirth \cite{voigtctx} to construct a graph that is 3-colourable but not 4-choosable. Though the base gadgets are identical, the final constructions are different. Interestingly, we note that most matchings in the correspondence assignment of the gadget have size only one.

\begin{proof}[Proof of Theorem \ref{ctx42}] We begin by constructing a gadget $H$ (shown in Figure \ref{fig:42}) with $V(H) = \{v_1, \dots, v_9\}$ and $E(H) = \{v_1v_3, v_1v_8, v_1v_5, v_1v_6, v_1v_4, v_2v_3, v_2v_9, v_2v_5, v_2v_7, v_2v_4, v_3v_8, v_3v_9, v_4v_6, v_4v_7, \\ v_5v_6,  v_5v_7, v_5v_8, v_5v_9\}$.

We define $L(v_1) = \{a\}$, and $L(v_2) = \{b\}$; $L(v_3) = \{a,b,1,2\}$ and $L(v_4) = \{a,b,3,4\}$; $L(v_6) = \{a,3,4,5\}$ and $L(v_7) = \{b,3,4,5\}$; $L(v_8) = \{a,1,2,6\}$ and $L(v_9) = \{b,1,2,6\}$; and finally $L(v_5) = \{a,b,5,6\}$.

 \begin{figure}[ht]
\begin{center}
\begin{tikzpicture}[scale=0.72, invisnode/.style={circle, draw=white, fill=white,  minimum size=1mm}, blcknd/.style={inner sep=1pt, minimum size=8pt, circle,draw,fill}]
\node[blcknd, label={above:\{a\}}] at (0,6) (v1) {\color{white}$v_1$};
\node[blcknd, label={below:\{b\}}] at (0,-6) (v2) {\color{white}$v_2$};
\node[blcknd, label={left:\{a,b,1,2\}}] at (-9.5,0) (v3) {\color{white}$v_3$};
\node[blcknd, label={right:\{a,b,3,4\}}] at (9.5,0) (v4) {\color{white}$v_4$};
\node[blcknd, label={left:{\{a,b,5,6\}}}] at (0,0) (v5) {\color{white}$v_5$};
\node[invisnode, label={right:\{a,3,4,5\}}] at (2.75,1.8) (v6f) {};
\node[invisnode, label={right:\{b,3,4,5\}}] at (2.75,-1.8) (v7f) {};
\node[invisnode, label={left:\{a,1,2,6\}}] at (-2.75,1.8) (v8f) {};
\node[invisnode, label={left:\{b,1,2,6\}}] at (-2.75,-1.8) (v9f) {};

\node[blcknd] (v6) at (2.9,1.5)  {\color{white}$v_6$};  
\node[blcknd] (v7) at (2.9,-1.5)  {\color{white}$v_7$};  
\node[blcknd] (v8) at (-2.9,1.5)  {\color{white}$v_8$};  
\node[blcknd] (v9) at (-2.9,-1.5)  {\color{white}$v_9$};

\node[] at (3.5,4.3) {a};
\node[] at (-3.5,4.3) {a};
\node[] at (3.5, -4.3) {b};
\node[] at (-3.5,-4.3) {b};

\node[] at (-1.9,-3.5) {b};
\node[] at (.3,-3.1) {b};
\node[] at (1.9,-3.5) {b};

\node[] at (-1.9,3.5) {a};
\node[] at (.3,3.1) {a};
\node[] at (1.9,3.5) {a};

\node[] at (5.2,0.6) {3,4};
\node[] at (5.2,-0.6) {3,4};
\node[] at (-5.2,0.6) {1,2};
\node[] at (-5.2,-0.6) {1,2};

\node[] at (3.4,0) {3,4};
\node[] at (-3.4,0) {1,2};

\node[] at (1,1) {5};
\node[] at (1,-1) {5};
\node[] at (-1,1) {6};
\node[] at (-1,-1) {6};

\draw[black] (v1)--(v3);
\draw[black] (v1)--(v8);
\draw[black] (v1)--(v5);
\draw[black] (v1)--(v6);
\draw[black] (v1)--(v4);
\draw[black] (v2)--(v3);
\draw[black] (v2)--(v9);
\draw[black] (v2)--(v5);
\draw[black] (v2)--(v7);
\draw[black] (v2)--(v4);
\draw[black] (v3)--(v8);
\draw[black] (v3)--(v9);
\draw[black] (v4)--(v6);
\draw[black] (v4)--(v7);
\draw[black] (v5)--(v6);
\draw[black] (v5)--(v7);
\draw[black] (v5)--(v8);
\draw[black] (v5)--(v9);
\draw[black] (v8)--(v9);
\draw[black] (v6)--(v7);
\end{tikzpicture}
\end{center}
\caption{The gadget used in the proof of Theorem \ref{ctx42}. The list of each vertex is indicated in braces beside the vertex. A colour $c$ is indicated beside an edge $e = uv$ iff $(u,c)(v,c) \in M_{uv}$.}
    \label{fig:42}
\end{figure}
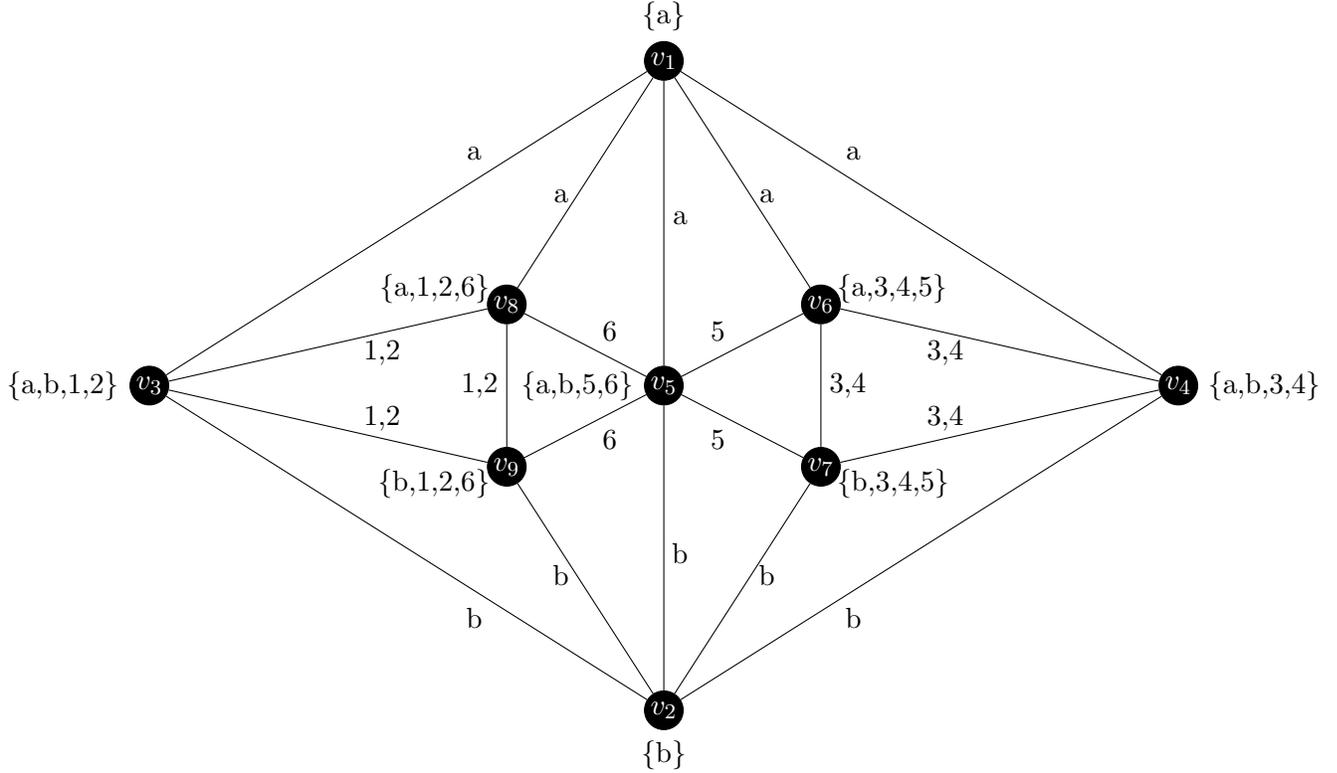

Finally, we define $M$ as follows. For each $u \in \{v_3, v_4, \dots, v_9\}$, set $M_{v_1u} = \{a\}$ and $M_{v_2u} = \{b\}$. For $u \in \{v_6, v_7\}$, set $M_{v_4u} = \{3,4\}$. For $u \in \{v_8,v_9\}$, set $M_{v_3u} = \{1,2\}$. Set $M_{v_6v_7} = \{3,4\}$, and $M_{v_8v_9} = \{1,2\}$. Finally, for $u \in \{v_6,v_7\}$, set $M_{v_5u} = \{5\}$ and for $u \in \{v_8, v_9\}$, set $M_{v_5u} = \{6\}$.  We claim that $H$ is not $(L,M)$-colourable: to see this, note that in every colouring $\phi$ of $H-v_5$, we have that $v_4$ is coloured either 3 or 4. This ensures that one of $v_6$ and $v_7$ receives colour 5. Symmetrically, one of $v_8$ and $v_9$ receives colour 6, and so there is no extension of $\phi$ to $v_5$.

Let $H_1, \dots, H_{16}$ be distinct copies of $H$, where in each copy of $H$ we define $(a,b)$ to be a distinct pair in $\{7,8,9,10\} \times \{11, 12, 13, 14\}$. Let $G$ be obtained from $H_1, \dots, H_{16}$ by identifying all copies of $v_1$ to a new vertex $u_1$ and all copies of $v_2$ to a new vertex $u_2$. Set $L(v_1) = \{7,8,9,10\}$ and $L(v_2) = \{11, 12, 13, 14\}$. Then $G$ is not $(L,M)$-colourable, as any choice of colour for $u_1$ and $u_2$ results in one of $H_1$ through $H_{16}$ having no colouring. 
\end{proof}

\begin{ack}
Thank you very much to both Alvaro Carbonero and Sophie Spirkl for helpful discussions. Thank you also to Ross Kang for pointing me in the direction of other results related to the correspondence colouring version of choosability with separation. In particular, an earlier version of this paper gave an analogous theorem to Theorem \ref{ctx42} for (3,1)-correspondence choosability, but this result was already known due to Zhu and Hell \cite{hell2008adaptable}.
\end{ack}
\bibliographystyle{siam}
\bibliography{bibliog}

\begin{thebibliography}{10}

\bibitem{berikkyzy20174}
{\sc Z.~Berikkyzy, C.~Cox, M.~Dairyko, K.~Hogenson, M.~Kumbhat, B.~Lidick{\`y},
  K.~Messerschmidt, K.~Moss, K.~Nowak, K.~F. Palmowski, et~al.}, {\em (4,
  2)-choosability of planar graphs with forbidden structures}, Graphs and
  Combinatorics, 33 (2017), pp.~751--787.

\bibitem{chen2017sufficient}
{\sc M.~Chen, Y.~Fan, Y.~Wang, and W.~Wang}, {\em A sufficient condition for
  planar graphs to be (3, 1)-choosable}, Journal of Combinatorial Optimization,
  34 (2017), pp.~987--1011.

\bibitem{chen2018choosability}
{\sc M.~Chen, K.-W. Lih, and W.~Wang}, {\em On choosability with separation of
  planar graphs without adjacent short cycles}, Bulletin of the Malaysian
  Mathematical Sciences Society, 41 (2018), pp.~1507--1518.

\bibitem{choi2016choosability}
{\sc I.~Choi, B.~Lidick{\`y}, and D.~Stolee}, {\em On choosability with
  separation of planar graphs with forbidden cycles}, Journal of Graph Theory,
  81 (2016), pp.~283--306.

\bibitem{dvovrak2021single}
{\sc Z.~Dvo{\v{r}}{\'a}k, L.~Esperet, R.~J. Kang, and K.~Ozeki}, {\em
  Single-conflict colouring}, Journal of Graph Theory, 97 (2021), pp.~148--160.

\bibitem{dvovrak2018correspondence}
{\sc Z.~Dvo{\v{r}}{\'a}k and L.~Postle}, {\em Correspondence coloring and its
  application to list-coloring planar graphs without cycles of lengths 4 to 8},
  Journal of Combinatorial Theory, Series B, 129 (2018), pp.~38--54.

\bibitem{erdos1979choosability}
{\sc P.~Erd\H{o}s, A.~L. Rubin, and H.~Taylor}, {\em Choosability in graphs},
  in Proc. West Coast Conf. on Combinatorics, Graph Theory and Computing,
  Congressus Numerantium, vol.~26, 1979, pp.~125--157.

\bibitem{hell2008adaptable}
{\sc P.~Hell and X.~Zhu}, {\em On the adaptable chromatic number of graphs},
  European Journal of Combinatorics, 29 (2008), pp.~912--921.

\bibitem{kierstead2015choosability}
{\sc H.~A. Kierstead and B.~Lidick{\`y}}, {\em On choosability with separation
  of planar graphs with lists of different sizes}, Discrete Mathematics, 338
  (2015), pp.~1779--1783.

\bibitem{kratochvil1998brooks}
{\sc J.~Kratochv{\'\i}l, Z.~Tuza, and M.~Voigt}, {\em Brooks-type theorems for
  choosability with separation}, Journal of Graph Theory, 27 (1998),
  pp.~43--49.

\bibitem{mirzakhani1996small}
{\sc M.~Mirzakhani}, {\em A small non-4-choosable planar graph}, Bull. Inst.
  Combin. Appl, 17 (1996), p.~3.

\bibitem{thomassen5LC}
{\sc C.~Thomassen}, {\em Every planar graph is 5-choosable}, Journal of
  Combinatorial Theory, Series B, 62 (1994), pp.~180--181.

\bibitem{vizing}
{\sc V.~Vizing}, {\em Vertex colourings with given colours}, Metody Diskret.
  Analiz.,  (1976), pp.~3--10.

\bibitem{voigt1993list}
{\sc M.~Voigt}, {\em List colourings of planar graphs}, Discrete Mathematics,
  120 (1993), pp.~215--219.

\bibitem{voigtctx}
{\sc M.~Voigt and B.~Wirth}, {\em On 3-colorable non-4-choosable planar
  graphs}, Journal of Graph Theory, 24 (1997), pp.~233--235.

\bibitem{skrekovski2001note}
{\sc R.~\v{S}krekovski}, {\em A note on choosability with separation for planar
  graphs}, Ars Combinatoria, 58 (2001), pp.~169--174.

\end{thebibliography}
\end{document}